\documentclass[11pt]{article}
\usepackage[all]{xypic}

\newtheorem{theorem}{Theorem}[section]
\newtheorem{corollary}[theorem]{Corollary}

\newtheorem{proposition}[theorem]{Proposition}
\newtheorem{definition}[theorem]{Definition}
\newtheorem{example}[theorem]{Example}
\newtheorem{remark}[theorem]{Remark}

\newenvironment{proof}{{\bf Proof}}{{\hfill $ \Box $}\vskip 4mm}

\def\Box{\mbox{$\sqcap\!\!\!\!\sqcup$}}

\def\ot{\otimes}
\def\ra{\rightarrow}

\def\eps{\varepsilon}

\def\bea{\begin{eqnarray*}}
\def\eea{\end{eqnarray*}}

\begin{document}
\title{Symmetric algebras of corepresentations and smash products}

\author{
S. D\u{a}sc\u{a}lescu$^1$, C. N\u{a}st\u{a}sescu$^{1,2}$ and L.
N\u{a}st\u{a}sescu$^{1,2}$\\
$^1$ University of Bucharest, Faculty of Mathematics and Computer
Science,\\
 Str. Academiei 14, Bucharest 1, RO-010014, Romania\\
$^2$ Institute of Mathematics of the Romanian Academy, PO-Box
1-764\\
RO-014700, Bucharest, Romania\\
 e-mail: sdascal@fmi.unibuc.ro,
 Constantin\_nastasescu@yahoo.com,\\
 lauranastasescu@gmail.com
}

\date{}
\maketitle

\begin{abstract}
We investigate Frobenius algebras and symmetric algebras in the
monoidal category of right comodules over a Hopf algebra $H$; for
the symmetric property $H$ is assumed to be cosovereign. If $H$ is
finite dimensional and $A$ is an $H$-comodule algebra, we uncover
the connection between $A$ and the smash product $A\# H^*$ with
respect to the
Frobenius and symmetric properties.
\end{abstract}

\section{Introduction and preliminaries}
We work over a basic field $k$. A finite dimensional algebra $A$
is called Frobenius if $A$ and its dual $A^*$ are isomorphic as
left (or equivalently, as right) $A$-modules. Frobenius algebras
arise in representation theory, Hopf algebra theory, quantum
groups, cohomology of compact oriented manifolds, topological
quantum field theory, the theory of subfactors and of extensions
of $C^*$-algebras, the quantum Yang-Baxter equation, etc., see
\cite{kadison}. A rich representation theory has been uncovered
for such algebras, see \cite{lam}, \cite{skr}. It was showed in
\cite{abrams1}, \cite{abrams2} that $A$ is Frobenius if and only
if it also has a coalgebra structure whose comultiplication is a
morphism of $A,A$-bimodules. This equivalent definition of the
Frobenius property has the advantage that it makes sense in any
monoidal category. The study of Frobenius algebras in monoidal
categories was initiated in \cite{muger}, \cite{street}, \cite{y},
and such objects have occurred in several contexts, for example in
the theory of Morita equivalences of tensor categories, in
conformal quantum field theory, in reconstruction theorems for
modular tensor categories, see more details and references in
\cite{fuchs}, \cite{fs}, \cite{muger2}; recent developments can be
also found in \cite{bt}.

 The representation
theory of Frobenius algebras uncovers several symmetry features,
for example there is a duality between the categories of left and
right finitely generated modules, and the lattices of left and
right ideals are anti-isomorphic. Among Frobenius algebras there
is a class of objects having even more symmetry. These are the
symmetric algebras $A$, for which $A$ and $A^*$ are isomorphic as
$A,A$-bimodules. The category of commutative symmetric algebras is
equivalent to the category of 2-dimensional topological quantum
field theories, see \cite{abrams1}. Symmetric algebras appear in
block theory of group algebras in positive characteristic, see
\cite[Chapter IV]{skr}.

It is not clear how one could define symmetric algebras in an
arbitrary monoidal category. Symmetric algebras in monoidal
categories with certain properties were first considered in
\cite{frs}, as structures related to correlation functions in
conformal field theories. In \cite{fuchs} symmetric algebras are
discussed in sovereign monoidal categories.
 In this paper we consider the monoidal
category ${\cal M}^H$ of right comodules (or corepresentations)
over a Hopf algebra $H$. If $A$ is an algebra in this category,
i.e. a right $H$-comodule algebra, then $A\in \, _A{\cal M}^H_A$,
i.e. $A$ is a left $(A,H)$-Doi-Hopf module and a right
$(A,H)$-Doi-Hopf module. On the other hand, $A^*$ is a right
$(A,H)$-Doi-Hopf module, but not necessarily a left
$(A,H)$-Doi-Hopf module; however $A^*$ has a natural structure of
a left $(A^{(S^2)},H)$-Doi-Hopf module, where $A^{(S^2)}$ is the
algebra $A$ with the coaction shifted by $S^2$, where $S$ is the
antipode of $H$. If $H$ is cosovereign, i.e. there exists a
character $u$ on $H$ such that $S^2(h)=\sum u^{-1}(h_1)u(h_3)h_2$
 for any $h\in H$, then $A\simeq A^{(S^2)}$ as comodule algebras,
 and this induces a structure of $A^*$ as an object in $_A{\cal
 M}^H_A$, where the left $A$-action is a deformation of the usual one  by
 $u$. Then it makes sense to consider when $A$ and
$A^*$ are isomorphic in this category; in this case we say that
$A$ is symmetric in ${\cal M}^H$ with respect to $u$, or shortly
that $A$ is $(H,u)$-symmetric. As ${\cal M}^H$ is a sovereign
monoidal category for such $H$, this is a special case of the
general concept of symmetric algebra considered in \cite{fuchs}.
In Section \ref{sectioncorepresentations} we give explicit
 characterizations of this property in ${\cal M}^H$.  We show that
the definition of symmetry depends on the character (i.e. on the
associated sovereign structure of ${\cal M}^H$).
 Also, we
use a modified version of the trivial extension construction to
give examples of $(H,u)$-symmetric algebras of corepresentations.
In the case where $H$ is involutory, i.e. $S^2=Id$, $H$ is
cosovereign if we take $u=\eps$, the counit of $H$, and in this
case it is clear that an $(H,\eps)$-symmetric algebra is also
symmetric as a $k$-algebra. However, we show that in general $H$
may be $(H,u)$-symmetric, without being symmetric as a
$k$-algebra.

Given a finite dimensional algebra $A$ in the category ${\cal
M}^H$, where $H$ is a finite dimensional Hopf algebra, one can
construct the smash product $A\# H^*$. Smash products are also
called semidirect products, since the group algebra of a
semidirect product of groups is just a smash product. Smash
product constructions are of great relevance since they describe
the algebra structure in a process of bosonization, which
associates for instance a Hopf algebra to a Hopf superalgebra. It
is proved in \cite{bergen} that $A$ is Frobenius if and only if so
is $A\# H^*$. On the other hand, we show in an example that such a
good connection does not hold for the symmetric property. In
Section \ref{sectionFrobenius} we show that if $A$ is a Frobenius
algebra in ${\cal M}^H$, then $A\# H^*$ is a Frobenius algebra in
${\cal M}^{H^*}$, but the converse does not hold.  In Section
\ref{sectionsymmetric} we uncover a good transfer of the symmetry
property between $A$ and $A\# H^*$, more precisely we show that
$A$ is $(H,\alpha)$-symmetric if and only if $A\# H^*$ is
$(H^*,g)$-symmetric, where $g$ and $\alpha$ are the distinguished
grouplike (or modular) elements of $H$ and $H^*$, provided that
$H$ is cosovereign by $\alpha$, and $H^*$ is cosovereign by $g$.

For basic concepts and notation on Hopf algebras we refer to
\cite{DNR}, \cite{radford}.

\section{Frobenius algebras and symmetric algebras of
corepresentations} \label{sectioncorepresentations}

Let $H$ be a Hopf algebra, and let $A$ be a finite dimensional
right $H$-comodule algebra, with $H$-coaction $a\mapsto \sum
a_0\ot a_1$. Then there exists an element $\sum_i a_i\ot h_i\ot
a_i^*\in A\ot H\ot A^*$ such that $\sum a_0\ot
a_1=\sum_ia_i^*(a)a_i\ot h_i$ for any $a\in A$; this element
corresponds to the $H$-comodule structure map of $A$ through the
natural isomorphism $A\ot H\ot A^*\simeq Hom(A, A\ot H)$. A right
$H$-comodule structure is induced on $A^*$ by
$$a^*\mapsto \sum_i a^*(a_i)a_i^*\ot S(h_i), \;\;\;\mbox{for any
}a^*\in A^*.$$ If we consider the left $H^*$-actions on $A$ and
$A^*$ associated with these right $H$-comodule structures, denoted
by $h^*\cdot a$ and $h^*\cdot a^*$ for $h^*\in H^*, a\in A, a^*\in
A^*$, we have \bea (h^*\cdot
a^*)(a)&=&\sum_ih^*(S(h_i))a^*(a_i)a_i^*(a)\\
&=&a^*(\sum_i h^*(S(h_i))a_i^*(a)a_i)\\
&=&a^*((h^*S)\cdot a)\eea so
\begin{equation} \label{formulaacth*}
(h^*\cdot a^*)(a)=a^*((h^*S)\cdot a)
\end{equation}
Moreover, $A^*\in {\cal M}^H_A$, with the usual right $A$-action;
this means that the $A$-module structure of $A^*$ is right
$H$-colinear. It is known (see \cite[Theorem 2.4]{dnn}) that the
following are equivalent: (1) $A\simeq A^*$ in ${\cal M}^H_A$; (2)
There exists a nondegenerate associative bilinear form $B:A\times
A\ra k$ such that $B(h^*\cdot a,b)=B(a,(h^*S)\cdot b)$ for any
$a,b\in A, h^*\in H^*$; (3) There exists a linear map
$\lambda:A\ra k$ such that $\lambda(h^*\cdot a)=h^*(1)\lambda(a)$
for any $a\in A,h^*\in H^*$, and ${\rm Ker}\,\lambda$ does not
contain a non-zero right ideal of $A$; (4) There exists a linear
map $\lambda:A\ra k$ such that $\lambda(h^*\cdot
a)=h^*(1)\lambda(a)$ for any $a\in A,h^*\in H^*$, and ${\rm
Ker}\,\lambda$ does not contain a non-zero subobject of $A$ in
${\cal M}^H_A$; (5) $A$ is a Frobenius algebra in the category
${\cal M}^H$. The connections between an isomorphism $\theta:A\ra
A^*$ as in (1), a bilinear map $B$ as in (2) and a linear map
$\lambda$ as in (3), (4) are given by $\theta(a)(b)=B(a,b)$,
$\lambda(a)=B(1,a)$,
$B(a,b)=\lambda(ab)$.\\

On the other hand, $A^*$ is also a left $A$-module in a natural
way, but in general $A^*$ is not an object of $_A{\cal M}^H$ (with
a similar compatibility condition for the $A$-action and
$H$-coaction). However, $A^*$ is an object in $_{A^{(S^2)}}{\cal
M}^H$, where $A^{(S^2)}$ is just the algebra $A$, with the
$H$-coaction shifted
by $S^2$, i.e. $a\mapsto \sum a_0\ot S^2(a_1)$.\\

Assume now that $H$ is a cosovereign Hopf algebra in the sense of
\cite{bichon}, i.e. there exists a character $u$ on $H$ (in other
words, $u$ is a grouplike element of the dual Hopf algebra $H^*$,
or equivalently, an algebra morphism from $H$ to $k$) such that
$S^2(h)=\sum u^{-1}(h_1)u(h_3)h_2$ for any $h\in H$; this is the
same with $(S^2)^*$ being an inner algebra automorphism of $H^*$
via $u$. Following \cite{bichon}, we say that $u$ is a sovereign
character of $H$. Then $f:A\ra A^{(S^2)}$, $f(a)=u^{-1}\cdot
a=\sum u^{-1}(a_1)a_0$, is an isomorphism of right $H$-comodule
algebras, and it induces an isomorphism of categories
$$F:\, _{A^{(S^2)}}{\cal M}^H \ra \, _A{\cal M}^H$$
where for $M\in \, _{A^{(S^2)}}{\cal M}^H$, $F(M)$ is just $M$,
with the same $H$-coaction, and $A$-action $*$ given by
$a*m=f(a)m$, for any $a\in A$ and $m\in M$. By restriction, this
induces an isomorphism of categories (and we denote it by $F$,
too)
$$F:\, _{A^{(S^2)}}{\cal M}^H_A \ra \, _A{\cal M}^H_A$$

Now $A^*\in \, _{A^{(S^2)}}{\cal M}^H_A$, so then $F(A^*)\in \,
_A{\cal M}^H_A$.

\begin{definition}
Let $H$ be a cosovereign Hopf algebra with $u$ as a sovereign
character. A finite dimensional right $H$-comodule algebra $A$ is
a symmetric algebra in the category ${\cal M}^H$ with respect to
$u$ if $F(A^*)\simeq A$ in the category $_A{\cal M}^H_A$. In this
case we simply say that $A$ is $(H,u)$-symmetric.
\end{definition}

Now we give equivalent characterizations of this property. The
next result can be derived from \cite[Proposition 4.6]{fuchs},
using the structure of duals in a category of corepresentations.
In our sketch of proof, explicit description is given for several
ways to describe symmetry of algebras of corepresentations.

\begin{proposition}
Let $A$ be a right $H$-comodule algebra, where $H$ is
a cosovereign Hopf algebra. Keeping the above notation, the following are equivalent.\\
(1) $A$ is $(H,u)$-symmetric.\\
(2) There exists a nondegenerate bilinear form $B:A\times A\ra k$
such that $B(b,ca)=B(bf(c),a)$, $B(b,a)=B(f(a),b)$, and
$B(b,h^*\cdot a)=B((h^*S)\cdot b,a)$ for any $a,b,c\in A, h^*\in H^*$.\\
(3) There exists a linear map $\lambda:A\ra k$ such that $\lambda
(ba)=\lambda (af(b))$ and $\lambda(h^*\cdot a)=h^*(1)\lambda(a)$
for any $a,b\in A,h^*\in H^*$, and also ${\rm Ker}\,\lambda$ does
not contain a non-zero right
ideal of $A$.\\
(4)  There exists a linear map $\lambda:A\ra k$ such that $\lambda
(ba)=\lambda (af(b))$ and $\lambda(h^*\cdot a)=h^*(1)\lambda(a)$
for any $a,b\in A,h^*\in H^*$, and also ${\rm Ker}\,\lambda$ does
not contain a non-zero
subobject of $A$ in ${\cal M}^H_A$.\\
More equivalent conditions can be added if we change right ideal
with left ideal in (3), and ${\cal M}^H_A$ with $_A{\cal M}^H$ in
(4).
\end{proposition}
\begin{proof}
We combine the proof of the equivalent characterizations of a
symmetric algebra in the category of vector spaces, see
\cite[Theorem 16.54]{lam}, and \cite[Theorem 2.4]{dnn}, recalled
above. Thus for $(1)\Leftrightarrow (2)$, if $\theta:A\ra F(A^*)$
is a linear map, then let $B:A\times A\ra k$ be the bilinear map
defined by $B(a,b)=\theta (b)(a)$. Then it is straightforward to
check that $\theta$ is left $A$-linear if and only if
\begin{equation} \label{B1}
B(b,ca)=B(bf(c),a) \;\;\;\mbox{ for any }a,b,c\in A,
\end{equation}
and $\theta$ is right $A$-linear if and only if
\begin{equation} \label{B2}
B(b,ac)=B(cb,a) \;\;\;\mbox{ for any }a,b,c\in A.
\end{equation}
We see that if (\ref{B1}) and (\ref{B2}) hold, then
$B(b,a)=B(b,a1)=B(bf(a),1)=B(f(a),b)$, thus
\begin{equation} \label{B3}
B(b,a)=B(f(a),b) \;\;\;\mbox{ for any }a,b\in A,
\end{equation}
Moreover, if (\ref{B1}) and (\ref{B3}) hold, then
$B(b,ac)=B(f(ac),b)=B(f(a)f(c),b)=B(f(a),cb)=B(cb,a)$, so
(\ref{B2}) holds. \\
We have that $\theta$ is $H$-colinear if and only if $B(b,h^*\cdot
a)=B((h^*S)\cdot b,a)$ for any $a,b\in A$, $h^*\in H^*$, and
$\theta$ is bijective if and only if $B$ is non-degenerate, thus
$(1)\Leftrightarrow (2)$ is clear.

For $(1)\Leftrightarrow (3)$, $\lambda$ and $B$ determine each
other by the relations $\lambda(a)=B(1,a)$ for any $a\in A$,
respectively $B(a,b)=\lambda (ba)$ for any $a,b\in A$.
\end{proof}

\begin{example}
1) If $H=kG$, the group Hopf algebra of a group $G$, then
$S^2=Id$, so $H$ is cosovereign with $\eps$ as a sovereign
character. A right $H$-comodule algebra is just a $G$-graded
algebra $A$, and $A$ is $(H,\eps)$-symmetric if and only if $A$ is
graded symmetric in the sense of \cite[Section 5]{dnn}.\\
2) More generally, if $H$ is an involutory Hopf algebra, i.e.
$S^2=Id$, then $H$ is obviously cosovereign with $\eps$ as a
sovereign character. In this case, if $A$ is a finite dimensional
algebra in ${\cal M}^H$, then $A^*\in \, _A{\cal M}^H_A$, so
$F(A^*)$ is just $A^*$, with the usual left and right $A$-actions.
Thus $A$ is $(H,\eps)$-symmetric if and only if $A^*\simeq A$ in
$_A{\cal M}^H_A$.
\end{example}

\begin{remark}
The definition of symmetry depends on the cosovereign character.
Thus it is possible that a cosovereign Hopf algebra $H$ has two
sovereign characters $u$ and $v$, and a right $H$-comodule algebra
$A$ is $(H,u)$-symmetric, but not $(H,v)$-symmetric. Indeed, let
$H=kC_2$, where $C_2=\{ e,g\}$ is a group of order 2 ($e$ is the
neutral element), and the characteristic of $k$ is $\neq 2$. Let
$A$ be a commutative $C_2$-graded division algebra with support
$C_2$; for example one can take $A=kC_2$. Then $H$ is involutory,
so it is a cosovereign Hopf algebra with two possible sovereign
characters $\eps=p_e+p_g$ and $u=p_e-p_g$, where $\{ p_e,p_g\}$ is
the basis of $H^*$ dual to the basis $\{ e,g\}$ of $H$. We have $u^2=\eps$, so $u^{-1}=u$.\\
It is easy to see that $A$ is $(H,\eps)$-symmetric, i.e. graded
symmetric in the terminology of \cite{dnn}, for example by using
the results of \cite{dnn2}, where the question whether any graded
division algebra is graded symmetric is addressed.\\
On the other hand, $A$ is not $(H,u)$-symmetric. Indeed, let
$\lambda:A\ra k$ be a linear map such that
$\lambda(ab)=\lambda(b(u\cdot a))$ for any $a,b\in A$. We have
$u\cdot a=a$ for any $a\in A_e$ (the homogeneous component of
degree $e$ of $A$), and $u\cdot a=-a$ for any $a\in A_g$. Then for
$b=1$ and $a\in A_g$ we get $\lambda(a)=0$, thus $\lambda(A_g)=0$.
Also, for $a,b\in A_g$ we obtain $\lambda(ab)=0$, and since
$A_gA_g=A_e$, this shows that $\lambda(A_e)=0$. Thus $\lambda$
must be zero.
\end{remark}

\begin{remark}\label{remarcaH4}
It is obvious that a graded symmetric algebra is symmetric as a
$k$-algebra. More general, if $H$ is involutory, then a
$(H,\eps)$-symmetric algebra is symmetric as a $k$-algebra.
However, for an arbitrary cosovereign Hopf algebra $H$ with
sovereign character $u$, if $A$ is $(H,u)$-symmetric, then $A$ is
not necessarily symmetric as a $k$-algebra, as we show in the
following example.

 Let $H=H_4$, the 4-dimensional Sweedler's Hopf
algebra. It is presented by algebra generators $c$ and $x$,
subject to relations
$$c^2=1, x^2=0, xc=-cx$$

The coalgebra structure is defined by
$$\Delta(c)=c\ot c, \Delta(x)=c\ot x+x\ot 1, \eps(c)=1, \eps(x)=0.$$

The antipode $S$ satisfies $S(c)=c, S(x)=-cx$, thus $S^2(x)=-x$.
Apart from $\eps$, $H$ has just one more character $\alpha$, given
by $\alpha(c)=-1, \alpha(x)=0$; it acts on the basis elements of
$H$ by
\begin{equation} \label{actionu}
\alpha\cdot 1=1,\; \alpha\cdot c=-c, \; \alpha\cdot x=x,\;
\alpha\cdot (cx)=-cx
\end{equation}

 $H$ is
cosovereign, and the only sovereign character is $\alpha$. We show
that the linear map $\lambda:H\ra k$,
$\lambda(1)=\lambda(c)=\lambda(cx)=0$, $\lambda(x)=1$ makes $H$ an
$(H,\alpha)$-symmetric algebra. Indeed, we first see by a
straightforward checking that $\lambda$ is right $H$-colinear (or
equivalently, left $H^*$-linear). Next, an easy computation using
$(\ref{actionu})$ shows that any element of the form
$ba-a(\alpha\cdot b)$ lies in the span of $1,c$ and $cx$, thus
$\lambda(ba)=\lambda(a(\alpha\cdot b))$ for any $a,b\in H$.

Finally, let $I$ be a left ideal of $H$ contained in ${\rm Ker}
\lambda$. Let $z=\delta 1+\beta c+\gamma cx\in I$, where
$\delta,\beta,\gamma\in k$. Then $cz=\delta c+\beta 1+\gamma x\in
I\subseteq <1,c,cx>$, so $\gamma=0$. Then $xz=\delta x-\beta cx\in
I\subseteq <1,c,cx>$, so $\delta=0$. Now $cxz=-\beta x\in
I\subseteq <1,c,cx>$, so $\beta$ must be zero, too. Thus $z=0$,
and $H$ is $(H,\alpha)$-symmetric. This will also follow from
Proposition \ref{Hopfsimetric}.

 On
the other hand, $H$ is not symmetric as a $k$-algebra. Indeed, if
$\lambda:H\ra k$ is a linear map such that
$\lambda(ab)=\lambda(ba)$ for any $a,b\in H$, then
$\lambda(cx)=\lambda(xc)=-\lambda(cx)$, so $\lambda(cx)=0$, and
$\lambda(x)=\lambda(ccx)=\lambda(cxc)=-\lambda(x)$, so
$\lambda(x)=0$. Thus the two-sided ideal $<x,cx>$ of $H$ is
contained in ${\rm Ker}\lambda$, showing that $H$ is not
symmetric. This can also be seen from a general result saying that
a Hopf algebra is symmetric as an algebra if and only if it is
unimodular (i.e. the spaces of left integrals and right integrals
in $H$ coincide) and $S^2$ is inner, see \cite{os}; for $H_4$ the
square of the antipode is inner, but the unimodularity condition
fails.
\end{remark}

Now we explain how examples of $(H,u)$-symmetric algebras in the
category ${\cal M}^H$ can be constructed, where $H$ is a
cosovereign Hopf algebra with sovereign character $u$. We recall
that for any algebra $A$ (in the category of vector spaces), and
any left $A$, right $A$-bimodule $M$, one can construct an algebra
structure on the space $A\oplus M$ with the multiplication defined
by $(a,m)(a',m')=(aa',am'+ma')$; this is called the trivial
extension of $A$ and $M$. The unit of this algebra is $(1,0)$. If
$M=A^*$ with the usual $A,A$-bimodule structure, the trivial
extension of $A$ and $A^*$ is simply called the trivial extension
of $A$, and it is a symmetric algebra, see \cite[Example
16.60]{lam}.

\begin{proposition} \label{trivialextension}
Let $A$ be a right $H$-comodule algebra, where $H$ is cosovereign
with sovereign character $u$. Then ${\cal E}(A)=A\oplus F(A^*)$,
with the direct sum structure of a right $H$-comodule,  and the
algebra structure obtained by the trivial extension of $A$ and the
$A,A$-bimodule $F(A^*)$, is a right $H$-comodule algebra which is
$(H,u)$-symmetric.
\end{proposition}
\begin{proof}
The multiplication of ${\cal E}(A)$ is given by
$$(a,a^*)(b,b^*)=(ab,a*b^*+a^*b)=(ab,(u^{-1}\cdot a)b^*+a^*b)$$ for any $a,b\in A$ and any
$a^*,b^*\in A^*$.

We first see that ${\cal E}(A)$ is a right $H$-comodule algebra.
Indeed, the $H$-coaction on ${\cal E}(A)$ is  $\rho:{\cal E}(A)\ra
{\cal E}(A)\ot H$, given by
$$\rho (a,a^*)=\sum (a_0,0)\ot a_1 + \sum (0,a^*_0)\ot a^*_1$$
Then \bea \rho((a,a^*)(b,b^*))&=&\rho (ab,a*b^*+a^*b)\\
&=&\sum ((ab)_0,0)\ot (ab)_1 + \sum (0,(a*b^*)_0)\ot (a*b^*)_1 +
\sum (0, (a^*b)_0)\ot (a^*b)_1\\
&=&\sum (a_0b_0,0)\ot a_1b_1 + \sum (0,a_0*b_0^*)\ot a_1b_1^*
+\sum
(0,a_0^*b_0)\ot a_1^*b_1\\
&=&(\sum (a_0,0)\ot a_1 + \sum (0,a^*_0)\ot a^*_1)(\sum (b_0,0)\ot
b_1 + \sum (0,b^*_0)\ot b^*_1)\\
&=& \rho (a,a^*)\rho (b,b^*)\eea

Let $\lambda:{\cal E}(A)\ra k$ be the linear map defined by
$\lambda(a,a^*)=a^*(1)$ for any $a\in A, a^*\in A^*$. Then
 \bea
\lambda (h^*\cdot (a,a^*))&=&(h^*\cdot
a^*)(1)\\
&=&a^*((h^*S)\cdot 1) \;\;\;\; (\mbox{by }(\ref{formulaacth*}))\\
&=&a^*((h^*S)(1)1)\\
&=&h^*(1)a^*(1)\\
&=&h^*(1)\lambda(a,a^*)\eea

Now \bea \lambda((a,a^*)(u^{-1}\cdot (b,b^*)))&=&\lambda
((a,a^*)(u^{-1}\cdot b,u^{-1}\cdot b^*))\\
&=&\lambda (a(u^{-1}\cdot b),(u^{-1}\cdot a)(u^{-1}\cdot b^*)+a^*(u^{-1}\cdot b))\\
&=&(u^{-1}\cdot b^*)(u^{-1}\cdot a)+a^*(u^{-1}\cdot b)\\
&=&b^*((u^{-1}S)\cdot (u^{-1}\cdot a))+a^*(u^{-1}\cdot b)\;\;\;\; (\mbox{by }(\ref{formulaacth*}))\\
&=&b^*(u\cdot (u^{-1}\cdot a))+a^*(u^{-1}\cdot b)\\
&=&b^*(a)+a^*(u^{-1}\cdot b)\\
&=&\lambda (ba,(u^{-1}\cdot
b)a^*+b^*a)\\
&=&\lambda ((b,b^*)(a,a^*))\eea

Finally, we see that ${\rm Ker}\lambda$ does not contain non-zero
right ideals. Indeed, if $\lambda((a,a^*){\cal E}(A))=0$, then
$b^*(u^{-1}\cdot a)+a^*(b)=0$ for any $b\in A, b^*\in A^*$. If we
take $b^*=0$, we get that $a^*(b)=0$ for any $b$, so $a^*=0$. Then
$b^*(u^{-1}\cdot a)=0$ for any $b^*$, so $u^{-1}\cdot a=0$,
showing that $a=0$.

 We conclude that $\lambda$ makes
${\cal E}(A)$ an $(H,u)$-symmetric algebra.
\end{proof}

We note that the previous result shows that any finite dimensional
algebra in the category ${\cal M}^H$, where $H$ is cosovereign via
$u$, is a subalgebra of an $(H,u)$-symmetric algebra, and also a
quotient of an $(H,u)$-symmetric algebra. The construction in
Proposition \ref{trivialextension} helps us to provide more
examples of algebras that are symmetric in categories of
corepresentations with respect to certain characters, but which
are not symmetric as $k$-algebras.

\begin{example}
Assume that $k$ has characteristic $\neq 2$, and let $H=H_4$ be
Sweedler's Hopf algebra. Let $A=k[X]/(X^2)$, a 2-dimensional
algebra, with basis $\{ 1,X\}$, and relation $X^2=0$. Let $c$ and
$x$ be the endomorphisms of the space $A$ such that $c(1)=1$,
$c(X)=-X$, $x(1)=0$, $x(X)=1$. Then
 $c^2=Id, x^2=0$ and $xc=-cx$. Moreover, $c$ is an algebra automorphism of
 $A$, and it is easy to check that $x(ab)=c(a)x(b)+x(a)b$ for any
 $a,b\in A$, thus $A$ is a left $H$-module algebra with the
 actions of $c$ and $x$ given by the endomorphisms above, i.e.
 $$c\cdot 1=1,\; c\cdot X=-X,\; x\cdot 1=0,\; x\cdot X=1.$$
 The dual space $A^*$ has a left $H$-action given by $(h\cdot
 a^*)(a)=a^*(S(h)\cdot a)$ for any $h\in H, a^*\in A^*$ and $a\in
 A$. If we consider the basis $\{ p_1,p_X\}$ of $A^*$ dual to the
 basis $\{1,X\}$, this action explicitly writes
 $$c\cdot p_1=p_1,\; c\cdot p_X=-p_X,\; x\cdot p_1=-p_X,\; x\cdot
 p_X=0.$$
 On the other hand, $A^*$ has usual left $A$-module structure
 given by
 $$1p_1=p_1, \; 1p_X=p_X, \; Xp_1=0,\; Xp_X=p_1,$$
 and usual right $A$-module structure given by
$$p_11=p_1, \; p_X1=p_X, \; p_1X=0,\; p_XX=p_1.$$
 $A$ is also a right comodule algebra over the dual
 Hopf algebra $H^*$. Since $H^*$ is cosovereign with sovereign character $c$ (via the isomorphism $H\simeq H^{**}$),
the associated left $A$-module structure on $F(A^*)$ is given by
$a*a^*=(c\cdot a)a^*$ for any $a\in A,a^*\in A^*$. Thus
$$1*p_1=p_1,\; 1*p_X=p_X,\; X*p_1=0,\; X*p_X=-p_1$$
The we can consider the algebra ${\cal E}(A)=A\oplus F(A^*)$,
whose basis is $\{ 1,X,p_1,p_X\}$, and multiplication induced by
$$X^2=p_1^2=p_X^2=p_1p_X=p_Xp_1=0$$
$$X*p_1=0,\; X*p_X=-p_1,\; p_1X=0,\; p_XX=p_1$$
If we denote $u=X$ and $v=p_X$, we can present ${\cal E}(A)$ by
generators $u,v$, subject to relations $u^2=v^2=0, vu=-uv$. The
left $H$-module structure of ${\cal E}(A)$ is given by $c\cdot
u=-u, c\cdot v=-v, x\cdot u=1, x\cdot v=0$. If we denote by $\{
P_1,P_c,P_x,P_{cx}\}$ the basis of $H^*$ dual to the standard
basis $\{ 1,c,x,cx\}$ of $H$, the right $H^*$-comodule structure
of ${\cal E}(A)$ is given by \bea u&\mapsto&u\ot P_1+c\cdot u\ot
P_c+ x\cdot u\ot P_x +(cx)\cdot
u\ot P_{cx}\\
&=&u\ot (P_1-P_c)+1\ot (P_x+P_{cx})\\
v&\mapsto &v\ot (P_1-P_c)\eea Since the Hopf algebra $H$ is
selfdual, a Hopf algebra isomorphism being given by $1\mapsto
P_1+P_c$, $c\mapsto P_1-P_c$, $x\mapsto P_x-P_{cx}$, $cx\mapsto
-P_x-P_{cx}$, we can regard $A$ as a right $H$-comodule algebra.
Summarizing, ${\cal E}(A)$ is the algebra with generators $u,v$,
relations
$$u^2=v^2=0, vu=-uv$$
and $H$-comodule structure given by
$$u\mapsto u\ot c-1\ot cx,\; v\mapsto v\ot c$$
By Proposition \ref{trivialextension}, ${\cal E}(A)$ is
$(H,\alpha)$-symmetric, where $\alpha=P_1-P_c$ is the
distinguished grouplike element of $H^*$.\\
On the other hand, ${\cal E}(A)$ is not symmetric as a
$k$-algebra. Indeed, if $\lambda:{\cal E}(A)\ra k$ is a linear map
such that $\lambda (zz')=\lambda (z'z)$ for any $z,z'\in {\cal
E}(A)$, then $\lambda(uv)=\lambda(vu)=-\lambda(uv)$, thus
$\lambda(uv)=0$. But the 1-dimensional space spanned by $uv$ is a
two-sided ideal of ${\cal E}(A)$, so ${\rm Ker}\lambda$ contains a
non-zero ideal.
\end{example}

\section{Frobenius smash products} \label{sectionFrobenius}

Let $A$ be an algebra in ${\cal M}^H$, where $H$ is a finite
dimensional Hopf algebra. Then $A$ is a left $H^*$-module algebra
and we can consider the smash product $A\# H^*$, which is an
algebra with multiplication given by
$$(a\# h^*)(b\# g^*)=\sum a(h^*_1\cdot b)\# h^*_2g^*$$
It is known that $A$ is a Frobenius algebra if and only if so is
$A\# H^*$, see \cite{bergen}.

On the other hand, $A\# H^*$ is an algebra in the category ${\cal
M}^{H^*}$, with the $H^*$-coaction induced by the comultiplication
of $H^*$, i.e. $a\# h^*\mapsto \sum a\# h^*_1\ot h^*_2$. The aim
of this section is to discuss the connection between $A$ being a
Frobenius  algebra in ${\cal M}^H$, and $A\# H^*$ being a
Frobenius  algebra in ${\cal M}^{H^*}$.

We consider the usual left and right actions of $H^*$ on $H$,
$h^*\rightharpoonup h=\sum h^*(h_2)h_1$ and $h\leftharpoonup
h^*=\sum h^*(h_1)h_2$, and the usual left and right actions of $H$
on $H^*$, denoted by $h\rightharpoonup h^*$ and $h^*\leftharpoonup
h$, where $h\in H$ and $h^*\in H^*$. $H$ also acts on $A\# H^*$ by
$h\rightharpoonup (a\# h^*)=a\# (h\rightharpoonup h^*)$.

We recall that a right (respectively left) integral in $H$ is an
element $t\in H$ such that $th=\eps (h)t$ (respectively
$ht=\eps(h)t$) for any $h\in H$, and a left integral on $H$ is an
element $T\in H^*$ such that $h^*T=h^*(1)T$ for any $h^*\in H^*$.

\begin{theorem} \label{teoremaFrobenius}
Let $H$ be a finite dimensional Hopf algebra, and let $A$ be a
finite dimensional right $H$-comodule algebra which is a Frobenius
algebra in the category ${\cal M}^H$. Then the smash product $A\#
H^*$ is a Frobenius algebra in the category ${\cal M}^{H^*}$.
\end{theorem}
\begin{proof}
Let $\lambda:A\ra k$ be a linear map whose kernel does not contain
non-zero right ideals of $A$, and such that $\lambda (h^*\cdot
a)=h^*(1)\lambda (a)$ for any $h^*\in H^*$ and $a\in A$.  Let $t$
be a non-zero right integral in $H$. Define a linear map
$\overline{\lambda}:A\#H^*\ra k$ such that
$$\overline{\lambda}(a\# h^*)=\lambda(a)h^*(t) \mbox{   for any
}a\in A, h^*\in H^*$$

We have that $\overline{\lambda}(h\cdot
z)=\varepsilon(h)\overline{\lambda}(z)$ for any $h\in H$ and $z\in
A\# H^*$. Indeed \bea \overline{\lambda}(h\cdot (a\# h^*))&=&\sum
\overline{\lambda}(a\#h^*_2(h)h^*_1)\\
&=& \sum h^*_2(h)\lambda(a)h^*_1(t)\\
&=&\lambda (a)h^*(th)\\
&=&\lambda(a)\varepsilon (h)h^*(t)\\
&=&\varepsilon (h)\overline{\lambda}(a\# h^*) \eea

We show that $Ker(\overline{\lambda})$ does not contain non-zero
subobjects of $A\# H^*$ in the category ${\cal M}_{A\#
H^*}^{H^*}$. Let $I$ be a right ideal of $A\# H^*$ which is also a
right $H^*$-subcomodule (or equivalently, invariant with respect
to the induced left $H$-action on $A\# H^*$) such that $I\subset
Ker(\overline{\lambda})$. We know that $(A\# H^*)^{co\; H^*}=A\#
1\simeq A$ and $A\# H^*/A$ is a right $H^*$-Galois extension, see
\cite[Example 6.4.8]{DNR}. Then $J=I^{co\; H^*}$ is a right ideal
of $A$ and the Weak Structure Theorem for Hopf-Galois extensions
shows that the map
$$J\ot_A(A\# H^*)\ra I,\;\; m\ot z\mapsto mz$$
is an isomorphism in the category ${\cal M}_{A\# H^*}^{H^*}$, see
\cite[Theorem 6.4.4]{DNR}. In particular $I=(J\# 1)(A\# H^*)=J\#
H^*$. Since $\overline{\lambda}(I)=0$, we see that $\lambda
(J)=0$, so $J=0$.
We conclude that $I$ must be zero.\\

\end{proof}

\begin{corollary}
A finite dimensional Hopf algebra $H$ is Frobenius in the category
${\cal M}^H$.
\end{corollary}
\begin{proof}
$k$ is a right $H^*$-comodule algebra in a trivial way, and it is
clear that $k$ is Frobenius in ${\cal M}^{H^*}$. By Theorem
\ref{teoremaFrobenius} we get that $k\# H^{**}$ is Frobenius in
${\cal M}^{H^{**}}$. Since $H^{**}\simeq H$ as Hopf algebras, we
see that $H$ is Frobenius in ${\cal M}^H$.
\end{proof}

The following example shows that the converse of Theorem
\ref{teoremaFrobenius} is not true.

\begin{example}
Let $A=A_0\oplus A_1$ be a superalgebra, i.e. a $C_2$-graded
algebra, which is Frobenius as an algebra, but not graded
Frobenius (i.e. it is not a Frobenius algebra in the category
${\cal M}^{kC_2}$ of supervector spaces). In this example we use
the additive notation for the operation of $C_2$. Examples of such
$A$ are given in \cite[Section 6]{dnn}; the trivial extension
associated to a finite dimensional algebra is one such example.
Let $\mu:A\ra k$ be a linear map whose kernel does not contain
non-zero left ideals of $A$. We define
$$\lambda:A\#(kC_2)^*\ra k,\; \lambda(a\# p_x)=\mu (a) \;\; \mbox{
for any }a\in A, x\in C_2$$ Then $\lambda(y\rightharpoonup (a\#
p_x))=\lambda (a\# p_{x-y})=\mu (a)=\eps (y)\lambda (a\# p_x)$.

On the other hand, ${\rm Ker}\lambda$ does not contain non-zero
left ideals of $A\# (kC_2)^*$. Indeed, assume that $\lambda ((A\#
(kC_2)^*)z)=0$, where $z=a\# p_0+b\# p_1$. Since $(c\#
p_0)z=ca_0\# p_0+cb_1\#p_1$, we have $0=\mu(ca_0)+\mu(cb_1)=\mu
(c(a_0+b_1))$ for any $c\in A$. Thus $\mu (A(a_0+b_1))=0$, showing
that $a_0+b_1=0$, and then $a_0=b_1=0$. Similarly, since $(c\#
p_1)z=ca_1\# p_0+cb_0\#p_1$, we obtain $a_1=b_0=0$. Thus $z=0$. We
conclude that $\lambda$ makes $A\# (kC_2)^*$ a Frobenius algebra
in the category ${\cal M}^{(kC_2)^*}$.
\end{example}

\section{Symmetric smash products} \label{sectionsymmetric}

The following shows that the good connection between a finite
dimensional right $H$-comodule algebra $A$ and the smash product
$A\# H^*$ being Frobenius does not work anymore for the symmetric
property.

\begin{example} \label{exemplusimetric}
Let $C_2=<c>=\{ e,g\}$ be the cyclic group of order 2, and let $A$
be a $C_2$-graded algebra which is symmetric and such that the
homogeneous component $A_e$ is not symmetric. For example one can
take the trivial extension $A=R\oplus R^*$ of a non-symmetric
algebra $R$, with the grading $A_e=R, A_g=R^*$. Then $A$ is a
right $kC_2$-comodule algebra, so we can consider the smash
product $A\# (kC_2)^*$. Denote by $\{ p_e,p_g\}$ the basis of
$(kC_2)^*$ dual to the basis $\{ e,g\}$ of $kC_2$. Then $1\# p_e$
is an idempotent in $A\# (kC_2)^*$ and it is easy to check that
$$(1\# p_e)(A\# (kC_2)^*)(1\# p_e)=A_e\# p_e\simeq A_e=R$$
Then $(1\# p_e)(A\# (kC_2)^*)(1\# p_e)$ is not a symmetric
algebra, so neither is $A\# (kC_2)^*$ by \cite[Exercise
16.25]{lam}.
\end{example}

In this section we discuss the connection between $A$ being a
symmetric algebra in ${\cal M}^H$ with respect to some character
of $H$, and $A\# H^*$ being a symmetric algebra in ${\cal
M}^{H^*}$ with respect to some character of $H^*$ (i.e. a
grouplike element of $H$).

Let $H$ be a finite dimensional Hopf algebra. Then there exists a
character $\alpha \in H^*$ such that $th=\alpha(h)t$ for any left
integral $t$ in $H$ and any $h\in H$; $\alpha$ is called the
distinguished grouplike element of $H^*$, and it also satisfies
$ht'=\alpha^{-1}(h)t'$ for any right integral $t'$ in $H$ and any
$h\in H$, see \cite[Section 10.5]{radford} or \cite[Section
5.5]{DNR}.

Similarly, there exists a distinguished grouplike element $g$ of
$H$, such that $Th^*=h^*(g)T$ for any left integral $T$ on $H$ and
any $h^*\in H^*$. We note that in \cite{radford}, $g^{-1}$ is
called the distinguished grouplike element of $H$; we prefer the
way we defined because $g$ will play the same role for $H$ as
$\alpha$ does for $H^*$. It is showed in \cite[Theorem
10.5.4]{radford} that for any left integral $t$ in $H$
\begin{equation} \label{deltat}
\Delta (t)=\sum S^2(t_2)g^{-1}\ot t_1
\end{equation}
Applying this for $H^*$ we see that for any left integral $T$ on
$H$ one has
\begin{equation} \label{deltaT}
\Delta(T)=\sum (T_2S^2)\alpha^{-1}\ot T_1
\end{equation}
and then for any $h\in H$ \bea T\leftharpoonup h&=& \sum
T_1(h)T_2\\
&=& ((T_2S^2)\alpha^{-1})(h)T_1\\
&=&\sum T_2(S^2(h_1))\alpha^{-1}(h_2)T_1\\
&=&\sum \alpha^{-1}(h_2)(S^2(h_1)\rightharpoonup T)\eea Thus for
any left integral $T$ on $H$ and any $h\in H$
\begin{equation} \label{Trighth}
T\leftharpoonup h=\sum \alpha^{-1}(h_2)(S^2(h_1)\rightharpoonup T)
\end{equation}

Now if $t$ is a left integral in $H$, then $S(t)$ is a right
integral in $H$ and  \bea \Delta
(S(t))&=&\sum S(t_2)\ot S(t_1)\\
&=& \sum S(t_1)\ot S(S^2(t_2)g^{-1}) \;\;\; \mbox{( by }
(\ref{deltat}))\\
&=& \sum S(t_1)\ot gS^2(S(t_2))\\
&=&S(t)_2\ot gS^2(S(t)_1)\eea We conclude that for any right
integral $t$ in $H$
\begin{equation} \label{deltatright}
\Delta(t)=\sum t_2\ot gS^2(t_1)\end{equation}

We also see that for a right integral $t$ in $H$ and $h\in H$ \bea
\sum t_1\ot ht_2&=&\sum \eps (h_1)t_1\ot h_2t_2\\
&=&\sum S(h_1)h_2t_1\ot h_3t_2\\
&=& \sum S(h_1)(h_2t)_1\ot (h_2t)_2\\
&=&\sum \alpha^{-1}(h_2)S(h_1)t_1\ot t_2\eea Thus we showed that
\begin{equation} \label{t1ht2}
\sum t_1\ot ht_2=\sum \alpha^{-1}(h_2)S(h_1)t_1\ot t_2
\end{equation}

\begin{theorem} \label{teoremasimetric}
Let $H$ be a finite dimensional Hopf algebra, and let $g$ and
$\alpha$ be the distinguished grouplike elements of $H$ and $H^*$.
We assume that  $S^2(h)=g^{-1}hg=\sum
\alpha^{-1}(h_1)\alpha(h_3)h_2$ for any $h\in H$. Then a right
$H$-comodule algebra $A$ is $(H,\alpha)$-symmetric if and only if
$A\# H^*$ is $(H^*,g)$-symmetric.
\end{theorem}
\begin{proof} We note that
\begin{equation} \label{S-2}
S^{-2}(h)=\sum \alpha (h_1)\alpha^{-1}(h_3)h_2
\end{equation}
for any $h\in H$.

 Assume that $A$ is $(H,\alpha)$-symmetric, and let
$\lambda:A\ra k$ such that $\lambda (ba)=\lambda
(a(\alpha^{-1}\cdot b))=\lambda(a\alpha^{-1}(b_1)b_0)$,
$\lambda(h^*\cdot a)=h^*(1)\lambda(a)$ for any $a,b\in A,h^*\in
H^*$, and also ${\rm Ker}\,\lambda$ does not contain a non-zero
right ideal of $A$. Let $t$ be a non-zero right integral in $H$
and define
$$\overline{\lambda}:A\#H^*\ra k, \;\; \overline{\lambda}(a\#
h^*)=\lambda (a)h^*(t)$$ as in the proof of Theorem
\ref{teoremaFrobenius}. This makes $A\#H^*$ a Frobenius algebra in
the category ${\cal M}^{H^*}$. In order to see that $A\#H^*$ is
symmetric in ${\cal M}^{H^*}$, it remains to show that
$\overline{\lambda}(zz')=\overline{\lambda}(z'(g^{-1}\rightharpoonup
z))$ for any $z,z'\in A\# H^*$, where $g^{-1}\rightharpoonup (a\#
h^*)=a\# (g^{-1}\rightharpoonup h^*)$. Indeed, we see that

\bea \overline{\lambda}((b\# g^*)(g^{-1}\rightharpoonup (a\#
h^*))&=&\overline{\lambda}((b\# g^*)(a\# (g^{-1}\rightharpoonup
h^*)))\\
&=&\sum \overline{\lambda}(b(g^*_1\cdot a)\#
g^*_2(g^{-1}\rightharpoonup h^*))\\
&=& \sum \lambda(bg^*_1(a_1)a_0)g^*_2(t_1)h^*(t_2g^{-1})\\
&=&\sum \lambda(ba_0)g^*(a_1t_1)h^*(t_2g^{-1})\\
&=&\sum \lambda(a_0\alpha^{-1}(b_1)b_0)g^*(a_1t_1)h^*(t_2g^{-1})\\
&=&\sum \lambda(a_0b_0)g^*(a_1b_1S(b_2)t_1)\alpha^{-1}(b_3)h^*(t_2g^{-1})\\
&=&\sum \lambda (((S(b_1)t_1)\rightharpoonup g^*)\cdot
(ab_0))\alpha^{-1}(b_2)h^*(t_2g^{-1})\\
&=&\sum ((S(b_1)t_1)\rightharpoonup
g^*)(1)\lambda(ab_0)\alpha^{-1}(b_2)h^*(t_2g^{-1})\\
&=&\sum g^*(S(b_1)t_1)\lambda(ab_0)\alpha^{-1}(b_2)h^*(t_2g^{-1})\\
&=&\sum \lambda
(ab_0)g^*(S(b_1)t_2)\alpha^{-1}(b_2)h^*(gS^2(t_1)g^{-1}) \;\;\; \mbox{ by } (\ref{deltatright})\\
&=&\sum \lambda (ab_0)g^*(S(b_1)t_2)\alpha^{-1}(b_2)h^*(t_1)\\
&=&\sum \lambda (ab_0)g^*(t_2)\alpha^{-1}(b_3)h^*(\alpha^{-1}(S(b_1))S(S(b_2))t_1) \;\;\; \mbox{ by } (\ref{t1ht2})\\
&=&\sum \lambda (ab_0)g^*(t_2)\alpha^{-1}(b_3)h^*(\alpha(b_1)S^2(b_2)t_1)\\
&=&\sum\lambda(ab_0)g^*(t_2)h^*(S^2(\alpha(b_1)\alpha^{-1}(b_3)b_2)t_1)\\
&=&\sum\lambda(ab_0)g^*(t_2)h^*(S^2(S^{-2}(b_1))t_1)\;\;\; \mbox{ by } (\ref{S-2})\\
&=&\sum\lambda(ab_0)g^*(t_2)h^*(b_1t_1)\\
&=& \sum\lambda(ab_0)g^*(t_2)h^*_1(b_1)h^*_2(t_1)\\
&=&\sum \lambda (a(h^*_1\cdot b))(h^*_2g^*)(t)\\
&=& \sum \overline{\lambda}(a(h^*_1\cdot b)\# h^*_2g^*)\\
&=& \overline{\lambda}((a\# h^*)(b\# g^*))\eea

Conversely, assume that $A\# H^*$ is $(H^*,g)$-symmetric, and let
$\mu:A\# H^*\ra k$ be a linear map whose kernel does not contain
non-zero right ideals of $A\# H^*$, and such that $\mu
(h\rightharpoonup z)=\varepsilon (h)\mu (z)$ and $\mu (zz')=\mu
(z'(g^{-1}\rightharpoonup z))$ for any $h\in H$ and any $z,z'\in
A\#H^*$. Let $T$ be a left integral on $H$ and define
$$\tilde{\mu}:A\ra k,\;\; \tilde{\mu}(a)=\mu (a\# T).$$

Let $a\in A$ and $h^*\in H^*$. We note that $g\rightharpoonup T$
is a right integral on $H$, see \cite[Proposition 5.5.4]{DNR}. Let
$z=a\# (g\rightharpoonup T)$ and $z'=1\#h^*$. Then \bea
zz'-z'(g^{-1}\rightharpoonup z)&=&(a\# (g\rightharpoonup T))(1\#
h^*)-(1\# h^*)(a\#T)\\
&=&a\#(g\rightharpoonup T)h^*-\sum (h^*_1\cdot a)\#h^*_2T\\
&=&a\#h^*(1)(g\rightharpoonup T)-\sum (h^*_1\cdot a)\#h^*_2(1)T\\
&=&h^*(1)a\# (g\rightharpoonup T)-(h^*\cdot a)\#T\eea Since $\mu
(zz')=\mu (z'(g^{-1}\rightharpoonup z))$, we get \bea
\tilde{\mu}(h^*\cdot a)&=&\mu ((h^*\cdot a)\#T)\\ &=&h^*(1)\mu
(a\# (g\rightharpoonup
T))\\
&=&h^*(1)\eps (g)\mu (a\# T)\\
&=&h^*(1)\tilde{\mu}(a)\eea

If $I$ is a subobject of $A$ in $_A{\cal M}^H$ contained in ${\rm
Ker}\tilde{\mu}$, then $\mu (I\# T)=0$. But $I\# T$ is a left
ideal of $A\# H^*$, so it must be zero. Then $I$ must be zero,
too.

 To show that
$A$ is symmetric in ${\cal M}^H$ it only remains to check that
$\tilde{\mu}(ba)=\tilde{\mu}(a(\alpha^{-1}\cdot b))$ for any
$a,b\in A$. This holds true since
\bea \tilde{\mu}(ba)&=& \mu (ba\# T)\\
&=&\mu ((b\# \varepsilon)(a\# T))\\
&=&\mu ((a\# T)(g^{-1}\rightharpoonup(b\# \varepsilon)))\\
&=&\mu ((a\# T)(b\# (g^{-1}\rightharpoonup\varepsilon)))\\
&=&\mu ((a\# T)(b\# \varepsilon))\\
&=&\sum \mu (a(T_1\cdot b)\# T_2)\\
&=&\sum \mu (aT_1(b_1)b_0\# T_2)\\
&=& \sum \mu (ab_0\# (T\leftharpoonup b_1))\\
&=&\sum \mu (ab_0\# \alpha^{-1}(b_2)(S^2(b_1)\rightharpoonup T)) \;\;\; \mbox{ by } (\ref{Trighth})\\
&=&\sum \eps (S^2(b_1))\alpha^{-1}(b_2)\mu (ab_0\# T)\\
&=&\sum \eps (b_1)\alpha^{-1}(b_2)\mu (ab_0\# T)\\
&=&\sum \alpha^{-1}(b_1)\mu (ab_0\# T)\\
&=& \mu (a(\alpha^{-1}\cdot b)\# T)\\
 &=&\tilde{\mu}(a(\alpha^{-1}\cdot b)) \eea
\end{proof}

\begin{remark}
(1) The conditions on $H$ in Theorem \ref{teoremasimetric} are
satisfied if $H$ is involutory and unimodular, and $H^*$ is
unimodular. Indeed, in this case the distinguished grouplike
elements are trivial, i.e. $\alpha=\eps$ and $g=1$. \\
For example, this happens if $H=kG$, where $G$ is a finite group.
Thus a finite dimensional $G$-graded algebra is graded symmetric
if and only if the smash product $A\# (kG)^*$ is symmetric in
${\cal M}^{(kG)^*}$ with respect to 1.\\
(2) In the case where the characteristic of $k$ is 0, it is known
that $H$ is involutory if and only if $H$ is semisimple, if and
only if $H$ is cosemisimple, see \cite[Theorem 16.1.2]{radford},
and in this situation $H$ and $H^*$ are always unimodular. Thus
Theorem \ref{teoremasimetric} applies to any semisimple Hopf
algebra in
characteristic 0.\\
(3) If $k$ has positive characteristic, Theorem
\ref{teoremasimetric} applies to any semisimple cosemisimple Hopf
algebra $H$. Indeed, it is known that any such $H$ is involutory,
see \cite[Theorem 3.1]{eg}.\\
(4) A Hopf algebra satisfying the conditions of Theorem
\ref{teoremasimetric} is not necessarily involutory, and it may be
not unimodular; take for example Sweedler's 4-dimensional Hopf
algebra.
\end{remark}

As a consequence of Theorem \ref{teoremasimetric} we obtain that
if a Hopf algebra $H$ is cosovereign by $\alpha$ and $H^*$ is
cosovereign by $g$, then $H$ is $(H,\alpha)$-symmetric. In fact,
we can prove that $H$ is $(H,\alpha)$-symmetric with less
assumptions.

\begin{proposition} \label{Hopfsimetric}
Let $H$ be a finite dimensional Hopf algebra which is cosovereign
with sovereign element $\alpha$, the distinguished grouplike
element of $H^*$. Then $H$ is $(H,\alpha)$-symmetric.
\end{proposition}
\begin{proof}
In order to use the notation we have already developed, it is more
convenient to show that if $H^*$ is cosovereign by $g$, then $H^*$
is $(H^*,g)$-symmetric. If $t$ is a right integral in $H$, by the
proof of Theorem \ref{teoremaFrobenius} (when we take $A=k$ and
identify $A\#H^*$ with $H^*$) we have that the linear map
$\lambda:H^*\ra k$, $\lambda(h^*)=h^*(t)$ is $H$-linear, and its
kernel does not contain nonzero subobjects of $H^*$ in ${\cal
M}^{H^*}_{H^*}$. On the other hand, for any $h^*,g^*\in H^*$ \bea
\lambda(g^*(g^{-1}\rightharpoonup h^*))&=&\sum
g^*(t_1)h^*(t_2g^{-1})\\
&=&g^*(t_2)h^*(gS^2(t_1)g^{-1})\;\;\; \mbox{ by } (\ref{deltatright})\\
&=&\sum h^*(t_1)g^*(t_2)\\
&=&(h^*g^*)(t)\\
&=&\lambda(h^*g^*)\eea so $\lambda$ makes $H^*$ an
$(H^*,g)$-symmetric algebra.
\end{proof}

\section{Passing to coinvariants} \label{coinvarianti}

If $A$ is a right $H$-comodule algebra which is Frobenius
(respectively symmetric) as an algebra, it is a natural question
to ask whether this property transfers to the subalgebra of
coinvariants $A^{co H}$. It is easy to see that such a transfer
does not hold. Indeed, let $A$ be the algebra from Example
\ref{exemplusimetric}, which is symmetric. $A$ is a
$kC_2$-comodule algebra, and its subalgebra of coinvariants
 is just $A_e$, which is not even Frobenius. The following shows
 that a good transfer occurs if $A$ is Frobenius in the category
 ${\cal M}^H$, provided $H$ is cosemisimple.

 \begin{proposition}
Let $H$ be a cosemisimple Hopf algebra. If $A$ is a right
$H$-comodule algebra which is Frobenius in the category ${\cal
M}^H$, then $A^{co H}$ is a Frobenius  algebra. If moreover, $H$
is involutory and $A$ is $(H,\eps)$-symmetric, then $A^{co H}$ is
symmetric.
 \end{proposition}
 \begin{proof}
Let $i:A^{co H}\ra A$ be the inclusion map, and let $i^*:A^*\ra
(A^{co H})^*$ be its dual. Since $i$ is a morphism of $A^{co
H},A^{co H}$- bimodules, then so is $i^*$. If $A$ is Frobenius in
${\cal M}^H$, let $\theta:A\ra A^*$ be an isomorphism in the
category ${\cal M}^H_A$. We show that $i^*\theta i:A^{co H}\ra
(A^{co H})^*$ is an isomorphism of right $A^{co H}$-modules, i.e.
$A^{co H}$ is Frobenius. In fact it is enough to show that
$i^*\theta i$ is injective; since $Im\; \theta i=(A^*)^{co H}$,
this is the same with showing that $i^*_{|(A^*)^{co H}}$ is
injective.

The left $H^*$-action on $A^*$ induced by the right $H$-coaction
is $(h^*\cdot a^*)(a)=\sum h^*S(a_1)a^*(a_0)$, for any $A\in A$.
Then $a^*\in (A^*)^{co H}$ if and only if $h^*\cdot a^*=h^*(1)a^*$
for any $h^*\in H^*$, and this means that $a^*(\sum
h^*S(a_1)a_0-h^*(1)a)=0$ for any $a\in A$ and any $h^*\in H^*$.
Since $S$ is bijective ($H$ is cosemisimple), we get that $a^*\in
(A^*)^{co H}$ if and only if $a^*$ vanishes on the subspace
$$V=< h^*(a_1)a_0-h^*(1)a\; |\; h^*\in H^*,a\in A \; >$$
Since
$$Ker\; i^*_{|(A^*)^{co H}}=\{ a^*\in (A^*)^{co H}\; |\; a^*(A^{co
H})=0\;\}$$ we see that $i^*_{|(A^*)^{co H}}$ is injective if and
only if $V+A^{co H}=A$. But this is indeed true, since for a left
integral $T$ on $H$ such that $T(1)=1$, one has $a=T\cdot a
-(T\cdot a-T(1)a)$. Moreover, $T\cdot a\in A^{co H}$, since
$h^*\cdot (T\cdot a)=(h^*T)\cdot a=h^*(1)Ta$ for any $h^*\in H^*$,
and obviously $T\cdot a-T(1)a\in V$. \\

For the second part we just have to note that $i^*\theta i$ is a
morphism of $A^{co H},A^{co H}$-bimodules since $\theta$ is an
isomorphism of $A,A$-bimodules; now the proof of the first part
works also in this case.
 \end{proof}

{\bf Acknowledgment}  The research was supported by the UEFISCDI
Grant PN-II-ID-PCE-2011-3-0635, contract no. 253/5.10.2011 of
CNCSIS. We thank Daniel Bulacu for very useful discussions about
Frobenius algebras.

\end{document}